\documentclass{imanum}

\usepackage{graphicx}
\usepackage{amsmath}
\usepackage{algorithm}
\usepackage{algorithmic}
\usepackage{mathrsfs}
\usepackage{amsfonts}
\usepackage{bbm}
\usepackage[hyphens]{url}
\usepackage{amssymb}
\renewcommand{\geq}{\geqslant}
\renewcommand{\leq}{\leqslant}

\DeclareMathAlphabet{\mathcal}{OMS}{cmsy}{m}{n}
\newcommand{\I}{\mathcal{I}}
\newcommand{\COLS}{\operatorname{\mathcal{C}}}
\newcommand{\ROWS}{\operatorname{\mathcal{R}}}

\newcommand{\opt}{\operatorname{opt}}

\renewcommand{\mathsf}[1]{{#1}}
\newcommand{\bfx}{\mathsf{x}}
\newcommand{\bfX}{\mathsf{X}}

\newcommand{\bfz}{\mathsf{z}}
\newcommand{\bfZ}{\mathsf{Z}}

\newcommand{\bisto}{\mathsf{B}}

\newcommand{\R}{\mathbb{R}}

\newcommand{\diag}{\operatorname{diag}}

\newcommand{\unit}{\mathbbm{1}}

\makeatletter
\def\@maketitle{%
  \newpage\setcounter{footnote}{1}
  \null
  \vspace*{12pt}%
  \parbox{132mm}{
  \begin{center}%
  \let \footnote \thanks
    {\fontsize{12}{15}\bf\selectfont \@title \par}%
    \vskip 1.2em%
    {
       {\fontsize{10}{12}\it\selectfont \@author\\[15.5pt]}
     }
          \vskip 1em%
  \end{center}
  }
  \vspace*{-.75pc}}

\def\ps@plain{%
  \def\@oddfoot{}
  \def\@evnhead{\CROPMARKSA}%
  \def\@oddhead{\CROPMARKSA}%
  \let\@mkboth\markboth
}

\begin{document}

\title{Solution of the optimal assignment problem by diagonal scaling algorithms\footnote{
This work has been supported in
part by the French National Research Agency (ANR) through the COSINUS
program (projects PETAL no ANR-08-COSI-009 and PETALH no
ANR-10-COSI-013) and by the Gaspard Monge Optimization Programme of 
Fondation Math\'ematique Jacques Hadamard and EDF.}
}

\shorttitle{Optimal assignment and diagonal scaling}

\author{%
{\sc
Meisam Sharify\thanks{Corresponding author. Email: Meisam.Sharify@manchester.ac.uk}}\\[2pt]
School of Mathematics, Alan Turing Building, The University of Manchester, Manchester, M139PL, UK\\[6pt]
{\sc
St\'ephane Gaubert\thanks{Email: Stephane.Gaubert@inria.fr}}\\[2pt]
INRIA  \& Centre de Math\'ematiques Appliqu\'ees, UMR 7641, Ecole Polytechnique, 91128 Palaiseau, France\\[6pt]
{\sc and}\\[6pt]
{\sc
Laura Grigori\thanks{Email: Laura.Grigori@inria.fr}} \\[2pt]
INRIA \& Laboratoire J.L. Lions, UMR 7598, Universite Pierre et Marie Curie, Paris, France
}

\shortauthorlist{M. Sharify, S. Gaubert and L. Grigori}

\maketitle

\begin{abstract}
{We show that a solution of the optimal assignment problem can be obtained
as the limit of the solution of an entropy maximization problem,
as a deformation parameter tends to infinity. This allows us
to apply entropy maximization algorithms to the optimal
assignment problem. In particular, the Sinkhorn algorithm leads
to a parallelizable method, which can be used as
a preprocessing to handle large dense optimal assignment
problems. This parallel preprocessing allows
one to delete entries which do not belong to optimal
permutations, leading to a reduced instance which becomes
solvable with limited memory requirements.}
{large scale optimal assignment problem; entropy maximization; 
matrix diagonal scaling; parallel computing; Sinkhorn iteration; Newton method.
} 
\end{abstract}

\section{Introduction}
\label{intro}

One of the most classical problems in combinatorial optimization is 
optimal assignment.
Several applications of this problem arise in different fields of applied
sciences
such as bioinformatics for protein structure alignment
problem~\cite{Holm93,Lin:2004}, VLSI design~\cite{Huang90}, image
processing and computer vision~\cite{Cheng96}, and the pivoting
problem in the solution of large linear systems of
equations~\cite{Olschowka96,duffkoster,li03superdist}.  Thus,
this problem has received considerable attention and several
algorithms have been proposed to solve it.

The first polynomial time algorithm to solve this problem was proposed
by~\cite{Kuhn55}. It works in $O(n^4)$ time, which
was improved to $O(n^3)$ by~\cite{Edmonds69} 
(see also~\cite{Dinic69}) where $n$ denotes the dimension of the input matrix.
In the sparse case,~\cite{Tarjan87} proposed an improved
algorithm which uses Fibonacci heaps for the shortest paths
computations. It runs in $O(n(m + n \log n))$ time where $m$ denotes the number of arcs. 
Several other algorithms have also been developed. We refer the
interested reader to the recent book of~\cite{Assignmentprobs}.

In this paper we exploit the connection between the optimal assignment problem 
and entropy maximization. The latter is well studied
in the field of convex optimization~\cite{entropybook}.
The main idea is to think of the optimal assignment problem 
as the limit of a deformation of an 
entropy maximization problem.
More precisely, given an $n\times n$ non-negative matrix $A=(a_{ij})$, let us look for 
an $n\times n$ bistochastic matrix $\bfX=(\bfx_{ij})$ maximizing the
relative entropy
\begin{equation}
\label{int_entropy}
J_p(\bfX):=
-\sum_{1\leq i,j\leq n} \bfx_{ij}(\log(\bfx_{ij}/a_{ij}^p) -1)  \enspace ,
\end{equation}

Here, $p$ is the deformation parameter. We will show in Section
\ref{sec:entmax} that when $p$ goes to infinity, the unique solution 
$X(p)=(x_{ij}(p))$ of the entropy maximization problem converges to a point $X(\infty)$ which
is of maximal entropy among the ones in the convex hull of the
matrices representing optimal permutations. In particular, if there is only
one optimal permutation, $X(p)$ converges to the matrix representing this
optimal permutation.
In Section~\ref{sec:maxent} we prove that, as $p\to \infty$,
\[
|\bfx_{ij}(p)- \bfx_{ij}(\infty)|= O(\exp(-cp)), \qquad \forall 1\leq i,j\leq n 
\] 
for some constant $c>0$.  This shows an exponential convergence to the optimal
solution when $p$ increases.

The maximal entropy matrix $\bfX(p)$
can be computed by any matrix scaling algorithm
such as Sinkhorn iteration~\cite{sinkhorn67} or Newton method~\cite{knight2007}. 
Subsequently, these iterative methods can be used to develop new algorithms 
to solve the optimal assignment problem
and related combinatorial optimization problems.

In Section~\ref{sec:conv}, we introduce an iterative method which is based on a
modification of Sinkhorn scaling algorithm, in which
the deformation parameter is slowly increased (this procedure
is reminiscent from simulated annealing, the parameter $p$
playing the role of the inverse of the temperature).
We prove that this
iteration, which we refer to as \textit{deformed-Sinkhorn iteration},
converges to a matrix whose entries that belong to the optimal
permutations are nonzero, while all the other entries are zero. An
estimation of the rate of convergence is also presented, 
but this appears to be mostly of theoretical interest
since in practice, the convergence of this algorithm
appears to be slow.

In Section~\ref{sec:fixp}, we investigate a preprocessing algorithm which can be used in 
the solution of large scale dense optimal assignment problem. This problem appears 
in several applications such as vehicle routing problem, object 
recognition and computer vision~\cite{auction2009}.
An application to cosmology (reconstruction
of the early universe) can be found in the work of~\cite{reconstruction2008}.
Models of large dense random assignment problems are also 
considered in~\cite[Ch.~VII]{spinglass} from the point of view of statistical physics.

Our preprocessing algorithm, is based on an iterative method that eliminates the 
entries not belonging to an optimal assignment.
This reduces the initial problem to a much smaller problem in terms of memory requirements.
This is illustrated in Figures~\ref{pic:dense} and~\ref{pic:sparse}.

\begin{figure}[t!]
\begin{minipage}{0.47\linewidth}
\centering
\includegraphics[scale=0.5]{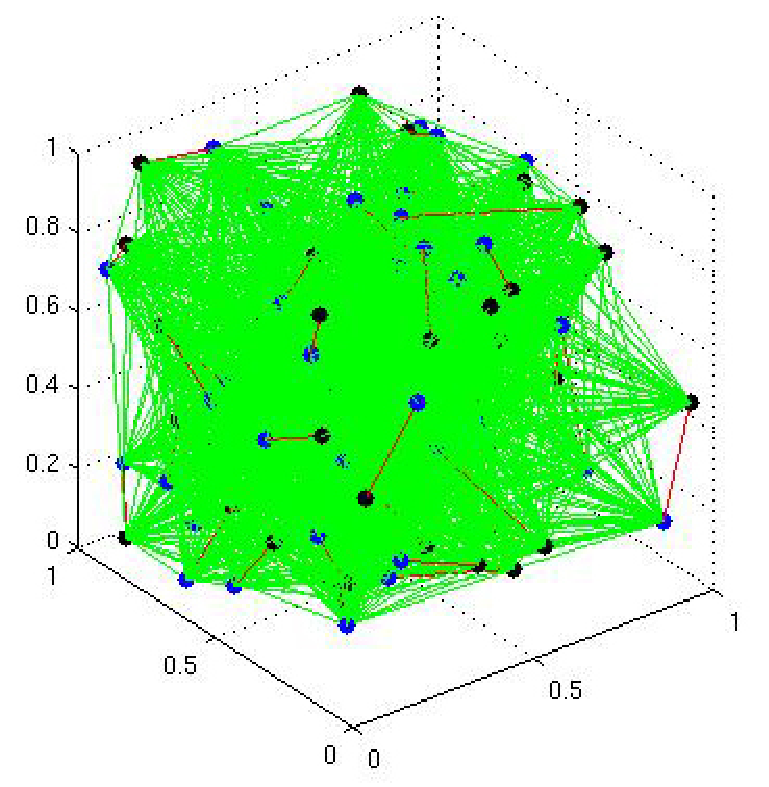}
\caption{Euclidean random assignment problem, see Section~\ref{subsec:sinkhexperiments} for more details.} 
\label{pic:dense}
\end{minipage} 
\begin{minipage}{0.47\linewidth}
\centering
\includegraphics[scale=0.5]{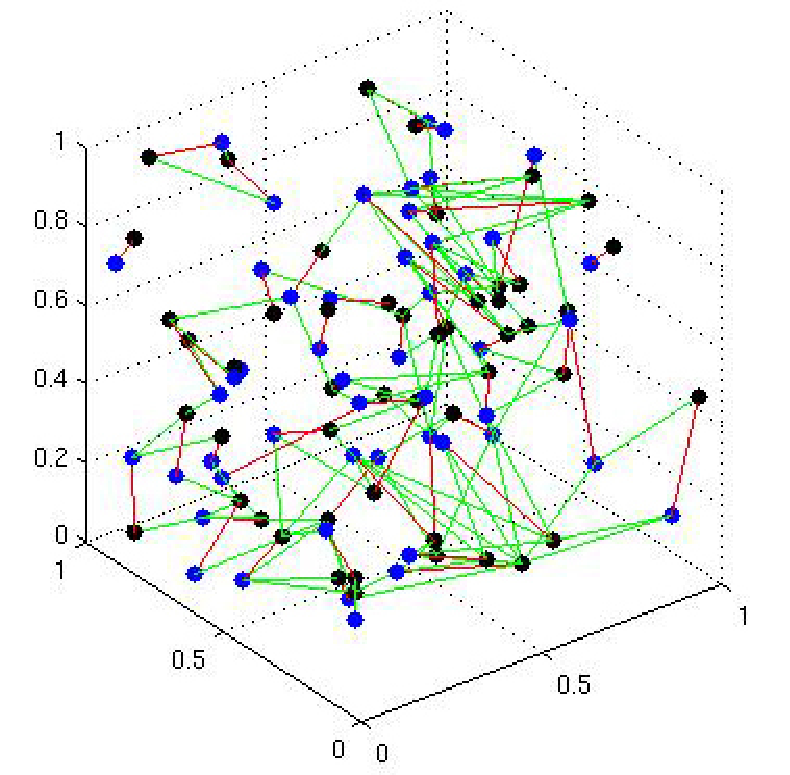} 
\caption{Reduced problem after applying the preprocessing algorithm.}
\label{pic:sparse}
\end{minipage} 
\end{figure}

The idea of this
algorithm is to take $p$ large enough, then apply a diagonal scaling algorithm 
to $A^{(p)}$ until convergence to a bistochastic matrix $\bfX$, and
finally delete the small entries of $\bfX$. 
Computing naively the exponential of $A^{(p)}$ 
would lead to numerical overflow for large values of $p$. 
However, in Section~\ref{sec:numins},
we shall see that it is possible to implement this iteration in a numerically stable way (with ``log-glasses'').
We also provide an approximate optimality certificate 
which can be used to check a posteriori that the result of the algorithm is optimal up to a given factor.
Our algorithm, presented in Section~\ref{sec:preprocalg}, assumes the existence of at least
one matching, since otherwise, Sinkhorn iteration may not converge.

In Section~\ref{sec:experminets}, we present the result of running the
preprocessing algorithm on several dense matrices from gallery of Matlab.
We consider two variants of the algorithm, one by using Sinkhorn iteration
as the diagonal scaling algorithm and the other one by using Newton iteration.
The advantage of Newton method is the speed of the convergence to bistochastic matrix. 
On the other hand, the advantage of Sinkhorn iteration is that, it can be efficiently
implemented in parallel~\cite{patrickduff08,duru:08}. 
So, the latter variant of the preprocessing algorithm can be used in the solution of very large dense optimal 
assignment problems, where the data cannot be stored in one machine.
In this way, the algorithm can run in parallel and reduces the size of the original problem 
and then the reduced problem can be solved by any classical method. 
For both variants, we show that the preprocessing algorithm can be efficiently used to decrease the size of the  
dense matrices, up to $99\%$.

\section{Entropy maximization and matrix scaling}
\label{sec:entmax}
In this section, after recalling some known facts concerning the connection between entropy maximization and matrix
scaling problems, we show (Theorem~\ref{thm:prop-converge1}) 
that the unique solution of a deformed entropy maximization problem i.e., 
an entropy maximization problem depending of a deformation parameter,
converges to the solution of an optimal assignment problem. We also determine
the convergence speed (Corollary~\ref{exponential_conv}).

\subsection{Background results}
The diagonal scaling problem can be generally defined as finding diagonal matrices $D_r$ and $D_c$ with positive
diagonal entries such that the scaled matrix $D_rAD_c$ has
prescribed row and column sums.
Due to the variety of its applications, this problem has been well 
studied~\cite{menonSchneider69,Brualdi,sinkhorn67}.
A comparison of the proposed algorithms to solve this problem, can be found in~\cite{MSchneider1990}.
A remarkable special case arises when the row and column sums of the matrix
$\bfX=D_rAD_c$ are required to be identically one, so that 
$\bfX$ is bistochastic.

A non-negative matrix, $A$, {\em has support} if it has a positive diagonal that is, 
there exists a permutation $\sigma$ such that $a_{i\sigma(i)}> 0, 1\leq i \leq n$.
Also, it {\em has total support} if every positive entry belongs to a diagonal.
The standard way to check whether a matrix has support is to compute 
its Dulmage-Mendelsohn decomposition~\cite{dum58,Pothen90}. However, this approach does not
lead to a parallel algorithm. 
We note that an alternative algorithm based on diagonal scaling has been proposed 
by~\cite{linial}.
This algorithm allows one to determine whether a $n\times n$ 
matrix has support after $n^2\log(n)$ Sinkhorn iterations.
It can be implemented in parallel. 

A non-negative square matrix $A$ is fully indecomposable if 
there does not exist permutation matrices $P$ and $Q$ such that
$PAQ$ is of the form 
\[
\begin{pmatrix} A_{1} & A_{2}\\ O & A_{3}\end{pmatrix}\enspace,
\]
where $A_{1}$ and $A_{3}$ are square matrices.
The following theorem provides a sufficient condition for the existence of a diagonal scaling.  
\begin{theorem}[\cite{sinkhorn67}]
Let $A$ be an $n\times n$ non-negative matrix with total support. Then there exist
diagonal matrices $D_r$ and $D_c$ such that $D_r A D_c$ is bistochastic.
Moreover, if $A$ is fully indecomposable, then $D_r$ and $D_c$ are unique up to a constant factor.
\end{theorem}

Now, consider the following optimization problem, which consists in
finding an $n\times n$ bistochastic matrix $\bfX=(\bfx_{ij})$ maximizing the
following relative entropy
\begin{equation}
\label{eq:pentropyopt}
\max_{\bfX\in \bisto_n} J_p(\bfX),\qquad J_p(\bfX):=\sum_{ij}\bfx_{ij}b_{ij} + 
p^{-1}S(\bfX ), \quad b_{ij}=\log a_{ij},
\end{equation}
where 
\[
\qquad S(\bfX):=-\sum_{ij}\bfx_{ij}\log \bfx_{ij} \enspace 
\]
is the entropy function, $p>0$ is a parameter, and $\bisto_n$ denotes the set of $n\times n$ bistochastic matrices.
We define $0\times (-\infty)\equiv 0$ in the context of the product $\bfx_{ij}b_{ij}$.

We shall assume that the matrix $A:=(a_{ij})$ has total support,
so that the diagonal matrices $D_r$ and $D_c$ are known to exist. 
We denote by $G(A):=\{(i,j)\mid a_{ij}>0\}$
the {\em pattern} (set of non-zero entries) of the matrix $A$.

The general relation between the entropy maximization and 
scaling problems is well known, see e.g.~\cite{MR986907} for an overview.

\begin{theorem}[Corollary of~{\cite[Th.~3.1]{borweinlewisnussbaum}} also~\cite{ando89}]
\label{prop-borweinlewisnussbaum}
Let $A$ be a matrix with total support. Then, the solution $\bfX(p)$  of the entropy 
maximization problem indicated in Equation~\ref{eq:pentropyopt}
is unique and it is characterized by the existence of two 
diagonal matrices $D_r$ and $D_c$, such that $\bfX=D_rAD_c$.
\end{theorem}

Thus, the characterization of the theorem shows that
$\bfX$ is obtained from the $p$th Hadamard power
$A^{(p)}:=(a^p_{ij})$ by a diagonal scaling. 
The previous theorem is a special case of Theorem~3.1 of~\cite{borweinlewisnussbaum},
which is established in a more general infinite dimensional setting
(for $p=1$; but the result for an arbitrary $p$ follows trivially from it).

\subsection{Convergence to optimal assignment}

We now study the convergence of $\bfX(p)$ as $p$ tends to infinity.
We shall consider the face $F$ of the polytope of bistochastic matrices consisting of the optimal solutions of the linear programming
formulation of the optimal assignment problem
\[
\max_{\bfx \in\bisto_n} \sum_{ij}\bfx_{ij}b_{ij} =\max_{\sigma\in \mathfrak{S}_n}\sum_i b_{i\sigma(i)}\enspace .
\]
\begin{theorem}
\label{thm:prop-converge1}
As $p$ tends to infinity, the matrix $\bfX(p)$ converges to the unique matrix
$\bfX^*$ maximizing the entropy among the ones that belong to the face
$F$ consisting of the convex hull of optimal permutation matrices.
In particular, if the solution
of the optimal assignment problem is unique, then $\bfX(p)$ converges
to the associated bistochastic matrix.
\end{theorem}
\begin{proof}
Since $\bfX(p)$ is the point of maximum of $J_p$,
\begin{eqnarray*}
J_p(\bfX(p))&&= \sum_{ij}\bfx_{ij}(p)b_{ij} +p^{-1}S(\bfX(p))\\
&&\geq J_p(\bfX^*)= \sum_{ij}\bfx_{ij}^*b_{ij} +p^{-1}S(\bfX^*)\\
&& = \max_{\sigma\in \mathfrak{S}_n}\sum_{i}b_{i\sigma(i)}+p^{-1}S(\bfX^*)
\end{eqnarray*}
Consider a sequence $(p_k)_{k\geq 1}$ converging to infinity, and 
assume that $\bfX(p_k)$ converges to some matrix $\bfZ$, which must
belong to $\bisto_n$. Setting $p=p_k$ in the previous inequality
and taking the limit as $k$ tends to infinity,
we get $\sum_{ij}\bfz_{ij}b_{ij}\geq  \max_{\sigma\in \mathfrak{S}_n}\sum_{i}b_{i\sigma(i)}$, which shows that $\bfZ$ belongs to the face $F$. 

Observe that
\[
p_k^{-1}(S(\bfX(p_k))-S(\bfX^*))=
\left(J_{p_k}(\bfX(p_k)) -J_{p_k}(\bfX^*)\right)+
\bigl(\sum_{ij}\bfx^*_{ij}b_{ij}-\sum_{ij}\bfx_{ij}(p_k)b_{ij}
\bigr)
\]
is the sum of two non-negative terms, because
$X(p_k)$ is a point of maximum of $J_{p_k}$, 
and $X^*\in F$ is a convex hull of matrices representing
optimal permutations. It follows that $S(\bfX(p_k))-S(\bfX^*)\geq 0$,
and so, if $Z$ is any accumulation point of $\bfX(p_k)$ as $k$ tends
to infinity, $S(Z)-S(\bfX^*)\geq 0$, showing that $Z$
is of maximal entropy among the matrices in $F$. Since the
entropy function is strictly convex, $X^*$ is the only point
with the latter property, and so every accumulation point of $\bfX(p_k)$
is equal to $X^*$, showing that $\bfX(p)$ converges to $X^*$ as $p\to\infty$.
\end{proof}

\begin{corollary}
If there is only one optimal permutation, then $\bfX(p)$ converges
to the corresponding permutation matrix.
\end{corollary}

\subsection{Speed of convergence}
\label{sec:maxent}
We have already shown in Theorem~\ref{thm:prop-converge1}
that the maximal entropy solution $\bfX(p)$ converges
as $p$ tends to infinity, to a matrix $\bfX(\infty)$ which is a convex hull
of optimal
permutation matrices. In particular, $\bfX(p)$ converges to
an optimal permutation matrix if the optimal permutation is unique.
Now, the question is how fast this convergence is. This is answered by the next
results.

Following~\cite{hardy},
we call {\em generalized Dirichlet series}
in a parameter $t$ a sum
\begin{align}
s:=\sum_{\alpha\in \mathbb{R}} c_\alpha t^\alpha \enspace,
\label{e-dirichlet}
\end{align}
where the coefficients $c_\alpha \in \mathbb{C}$ 
are such that $\operatorname{supp} s:=
\{\alpha \in \mathbb{R}\mid c_\alpha\neq 0\}$ is either a finite
(possibly empty) set or a denumerable set having $\infty$ as the only accumulation point (in particular, a finite number of monomials with negative powers of $t$ may appear in the series).

\begin{theorem}\label{th-dirichlet}
Assume that the matrix $A$ has total support.
Then, every entry $\bfx_{ij}(p)$ is given by a generalized Dirichlet series 
in the parameter $t=\exp(-p)$ that is absolutely convergent in
some punctured disk $0<|t|<\delta$.
\end{theorem}
This field was used in particular 
by Akian, Bapat and Gaubert~\cite{ABG96} to address
a somehow related asymptotic problem, concerning the Perron eigenvector of the matrix $A^{(p)}$ as $p\to\infty$. The corresponding field of {\em formal} 
generalized Dirichlet series is also a useful tool in tropical
geometry, as pointed out by~\cite{Markwig}. 
Before proving this theorem, we derive the following
corollary.
\begin{corollary}
\label{exponential_conv}
Assume that the matrix $A$ has total support.
Then, there exists a positive constant $c$ such that,
\[
|\bfx_{ij}(p)- \bfx_{ij}(\infty)|= O(\exp(-cp))
\]
holds for all $1\leq i,j\leq n$, as $p\to \infty$.
\end{corollary}
\begin{proof}
We already showed that $\bfx_{ij}(p)$ converges to $\bfx_{ij}(\infty)$
as $p\to \infty$. It follows that the leading monomial of the Dirichlet
series expansion of $\bfx_{ij}(p)$ is $\bfx_{ij}(\infty)$, and that the
next monomial is necessarily of the form $d_{ij}\exp(-c_{ij}p)$
for some $c_{ij}>0$ and $d_{ij}\in \mathbb{R}$ (we adopt the convention
that $d_{ij}=0$ if $\bfx_{ij}(p)$ is constant near $p=\infty$).
Then, it suffices to take for $c$ the minimum of all the $c_{ij}$ obtained in this way.
\end{proof}

The proof of Theorem~\ref{th-dirichlet} uses a model theory 
result of~\cite{vandendries98},
who constructed a o-minimal expansion
of the field of real numbers which is such that the definable
functions in one variable correspond to generalized
Dirichlet series that are absolutely convergent in a punctured
disk. We briefly recall their construction, referring the reader
to the monography~\cite{vandendriescambridge} for more background
on o-minimal structures, and in particular
for the definition of the notions used here.

For any $m\geq 1$, let $t_1,\dots, t_m$ be commuting
variables, and consider a formal series
\[
F= \sum_\alpha c_\alpha t^\alpha , \; \alpha=(\alpha_1,\dots,\alpha_m),\qquad
t^\alpha:= t_1^{\alpha_1}\dots t_m^{\alpha_m}  \enspace,
\]
where the multi-index $\alpha$ ranges over $[0,\infty)^m$. 
The support of $F$ is now $\operatorname{supp} F:=\{\alpha \mid c_\alpha\neq 0\}$.
We denote by $\mathbb{R}[[\{t_1,\dots,t_m\}^+]]$ the $\mathbb{R}$-algebra
of formal series the support of which is included
in a Cartesian product $S_1\times \dots\times S_m$ where
every $S_i$ is either a finite subset of $\mathbb{R}$ or a denumerable
subset of $\mathbb{R}$ having $\infty$ has the only accumulation
point. For each $r=(r_1,\dots,r_m)$ with $0<r_i<\infty$ for $i=1,\dots,m$,
we set $\|F\|_r := \sum_\alpha |c_\alpha|r^\alpha$, and denote
by $\mathbb{R}\{\{t_1,\dots,t_m\}^+\}_r$ the subalgebra of $\mathbb{R}[[\{t_1,\dots,t_m\}^+]]$
consisting of those $F$ such that $\|F\|_r<\infty$. Then,
we denote by $\mathbb{R}_{\text{an}+}$ 
the expansion of the real ordered field 
$(\mathbb{R},<,0,1,+,-,\cdot)$ by the collection
of all functions
$f: \mathbb{R}^m \to \mathbb{R}$, with $m\in \mathbb{N}$,
such that $f$ is $0$ outside $[0,1]^m$ and that it is
given on $[0,1]^m$ by a power series $F\in \mathbb{R}\{\{t_1,\dots,t_m\}^+\}_r$
for some $r=(r_1,\dots,r_m)$ with $r_1>1,\dots,r_m>1$. 
The following theorem follows from~\cite{vandendries98}.
\begin{theorem}[See Theorem~B and~\S~10, Paragraph~2 in~{\cite{vandendries98}}]
\label{th-B+}
Let $\epsilon>0$ and let $f: (0,\epsilon)\to \mathbb{R}$ be definable
in $\mathbb{R}_{\text{an}+}$. Then, there exists a generalized
Dirichlet series $F\in \mathbb{R}\{\{t\}^+\}$ in a single variable
$t$ that is absolutely convergent in a punctured disk $0<|t|<\delta$
for some $\delta<\epsilon$, such that $f(t)=F(t)$ holds
for all $0<|t|<\delta$.
\end{theorem}
Actually, Theorem~B in~\cite{vandendries98} deals with a larger o-minimal
structure, $\mathbb{R}_{\text{an}*}$, in which it is
only required that every set $S_i$ is well ordered in
the above construction. The latter theorem shows that every definable
function in one variable in this larger structure
coincides with a Hahn series (series
with well ordered support) that is absolutely convergent
in a punctured disk. However, it is remarked in Section~10
of~\cite{vandendries98} that the same statement remains
true if ``well ordered'' is
replaced by ``finite or denumerable with $\infty$ as
the only accumulation point'' in the construction
of $\mathbb{R}_{\text{an}*}$ and if Hahn series
are replaced by generalized Dirichlet series.

\begin{proof}[of Theorem~\ref{th-dirichlet}]
We make the change of variable $t=\exp(-p)$, and
we will write $\bfX'(t)$ for $\bfX(-\log t)$.

Since the $n\times n$ matrix $A$
is non-negative and has total support, so does $A^{(p)}$ for all $0<p<\infty$,
and so, the solution
$\bfX(p)$ of the entropy maximization problem
is the only matrix $\bfX$ such that there
there exist
diagonal $n\times n$ matrices $D_r$ and $D_c$ with positive diagonal entries such that
$A^{(p)}=D_r \bfX D_c$. Every non-zero entry of the matrix
$A^{(p)}$ can be written as $a_{ij}^p = t^{-\log a_{ij}}$. In particular,
the function $t \mapsto a_{ij}^p$ belongs to $\mathbb{R}[[\{t\}^+]]_r$ for
every $r\in (0,\infty)$. It follows that 
for every $r\in (0,\infty)$, the function
 $f:(0,r)\to \mathbb{R}, t \mapsto \bfX'(t)$ is definable in the structure $\mathbb{R}_{\text{an+}}$. Hence,
by Theorem~\ref{th-B+}, every entry of $\bfX'(t)$
has an expansion as a generalized Dirichlet series in the variable
$t$ and this series is absolutely convergent in a punctured disk $0<|t|<\delta$.
\end{proof}

\begin{remark}
\textnormal{The formulation~\eqref{eq:pentropyopt} is somehow reminiscent of 
interior point methods, in which the entropy $S(X)=-\sum_{ij}x_{ij}\log x_{ij}$ is 
replaced by a log-barrier function (the latter would be $\sum_{ij}\log x_{ij}$ in the present setting). 
The present $X(p)$ thought of as a function of $p\to\infty$ is 
analogous to the {\em central path}, and as does the central path, $X(p)$ converges 
to a face containing optimal solutions. However, the entropy $S(X)$ does not satisfies the 
axioms of the theory of self-concordant barriers 
on which the analysis of interior point methods is based. Indeed, the speed of 
convergence in $O(\exp(-cp))$ appears to be of a totally different nature by 
comparison with the speed of $O(1/p)$ observed in interior point methods~\cite{nesterovnemirovski}.}
\end{remark}
\begin{example}
\textnormal{
The constant $c$ appearing in Corollary~\ref{exponential_conv} can be
small if there are several nearly optimal permutations, and then
a large value of $p$ may be needed to approximate $X(\infty)$. However,
in such cases, a much smaller value of $p$ turns out to be enough
for the method described in the next sections, the aim of which is to eliminate
a priori entries not belonging to (nearly) optimal permutations.
This is illustrated by the following matrix, in which the 
identity permutation is optimal, and the transposition $(1,2)$
is nearly optimal:
\[
A=
\begin{pmatrix}
1 &  0.99 &  0.99 \\
0.99 &   1 &  1/3\\
0.25  &  0.5 &     1  
\end{pmatrix}
\enspace .
\]
For $p=10$, we have the following matrix, 
the significant entries of which indicate precisely the
optimal and nearly optimal permutations:
\[
\begin{pmatrix}
0.5195148 &  0.4595136  & 0.0210196 \\
0.4804643 &  0.5195864  & 0.0000004 \\
0.0000209 &  0.0209000  & 0.9789800  
\end{pmatrix} \enspace .
\]
The convergence of $X(p)$ to $X(\infty)$ is illustrated in Figure~\ref{centeral_path1}. 
Observe that the graph of $\log x_{ij}(p)$ as a function of $p$ is approximately
piecewise affine. In fact, each piece corresponds to a monomial 
in the generalized Dirichlet series expansion~\eqref{e-dirichlet}.
The path $p\mapsto X(p)$ converges quickly to the face
containing the two nearly optimal permutations and slowly to the unique
optimal permutation.}
\begin{figure}[t!]
\centering
\includegraphics[scale=0.45]{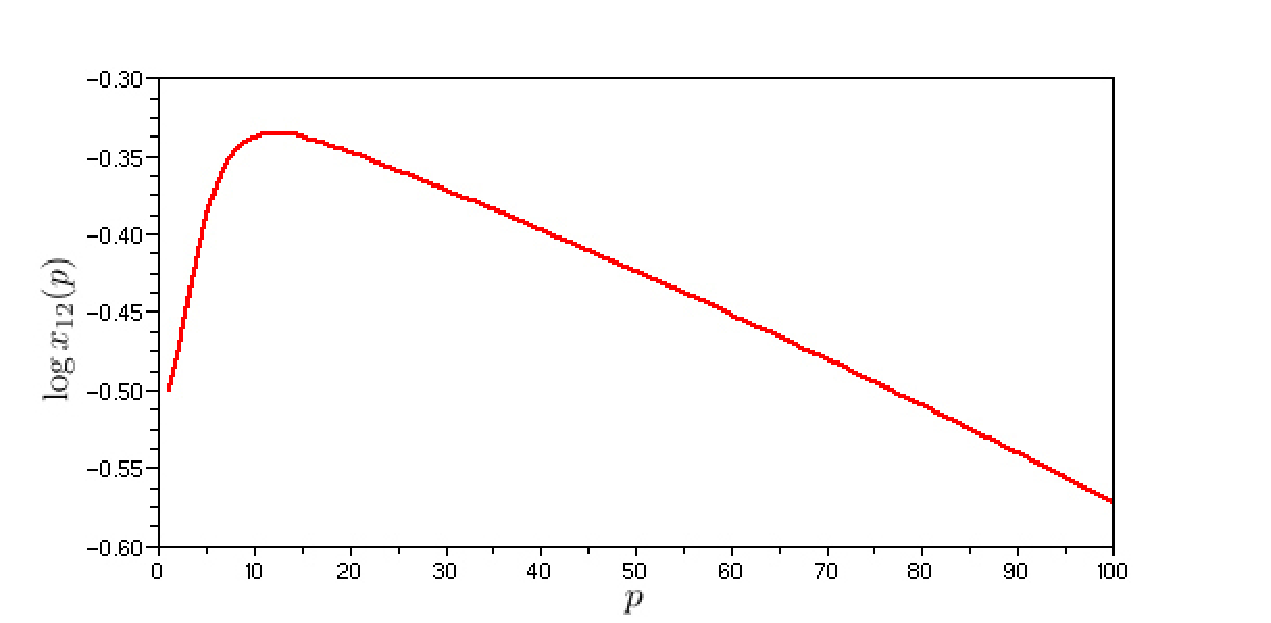} 
\caption{The variation of $\log_{10} x_{12}(p)$ as a function of $p$.} 
\label{centeral_path1}
\end{figure}
\end{example}

\begin{remark}
\textnormal{
Finding an explicit formula for the speed of convergence $c$ appears to be an interesting combinatorial problem (which is beyond the scope of this paper).}
\end{remark}

\section{Deformed Sinkhorn iteration}
\label{sec:conv}

In this section, we consider the Sinkhorn iteration, which is probably
the most classical way of 
computing the diagonal matrices $D_r,D_c$ of Theorem~\ref{prop-borweinlewisnussbaum}, leading to the solution of the entropy maximization problem.
We develop a ``path following method'' in which
the value of $p$ is gradually increased in the course
of Sinkhorn iterations.
We prove that if the matrix $A$ has support ($A$ has support if it has a positive diagonal), 
and if the growth of $p$ is moderate enough, then the sequence of matrices produced by the 
algorithm converges to a point which belongs to the face generated by optimal permutations.
The results of this section leads to an algorithm to compute the 
optimal assignment which we refer to as~\textit{deformed Sinkhorn iteration}.
This algorithm is mostly interesting from the theoretical point of view since the value of $p$ 
increases slowly in the course of the algorithm which yields a slow convergence to the solution. 

\subsection{Sinkhorn iteration}
\label{subsec:psinkh}
A simple way to compute the diagonal matrices $D_r,D_c$ is Sinkhorn iteration~\cite{sinkhorn67}.
This algorithm starts from a given matrix $A$, divides
every row by its sum, then every column of the new matrix
by its sum, and so on, until the matrix obtained in this way converges to
a bistochastic matrix. 
The advantage of this algorithm is that it can be efficiently implemented in 
parallel~\cite{patrickduff08} and it can be applied to any non-negative matrix which has
at least one nonzero permutation. The disadvantage is that, 
it is generally slower than other methods.

Recall first that the open cone $C=\{x\in \mathbb{R}^n: x_i>0,\forall i\}$
consisting of positive vectors of $\mathbb{R}^n$ 
is equipped with Hilbert's projective metric, defined by
\[
d(x,x')=\log \max_{i,j}\frac{x_ix'_j}{x'_ix_j}
\]
Note that $d(x,x')$ is zero if and only if the vectors $x$ and $x'$ are proportional. 
We refer to~\cite[\S~6]{Bapat97} for more background. In particular, if $A$ is a 
positive matrix, a theorem of Birkhoff shows that the map $x\mapsto Ax$ is a contraction 
in Hilbert's projective metric, with a contraction rate
\[
\kappa(A):=\sup\{\frac{d(Ay,Ay')}{d(y,y')}:y,y'\in C,y,y'\text{\ non proportional}\}
= 
\frac{\theta(A)^{1/2}-1}{\theta(A)^{1/2}+1}  \enspace ,
\]
where
\[
\theta(A)=\exp\sup\{d(Ay,Ay'):y,y'\in C\}=\max_{i,j,p,l}\frac{a_{ir}a_{jl}}{a_{jr}a_{il}}
\]
The following result is a consequence of this theorem.

\begin{proposition}[\cite{Frank89}]
For a positive matrix $A$, the global rate of convergence of Sinkhorn iteration is bounded above by $\kappa(A)^2$.
\end{proposition}

This general bound is applicable only for positive matrices and it can be coarse in practice.
We shall use this contraction rate to prove the convergence 
of our \textit{deformed Sinkhorn iteration} in Section~\ref{sec:deformed_conv}.

It is proved by~\cite{soules} that the rate of convergence of Sinkhorn 
algorithm is always linear when the 
input matrix, $A$, has total support. He defines, $\bfX^{k+1}=f(\bfX^k)$ where, $X^0=A$ and $f$ is 
an operator 
which divides each row by its sum and then divides each column by its sum. He also defines 
a scalar,
\[
\xi=\limsup_k\frac{\|\bfX^{k+1}-\bfX^*\|}{\|\bfX^{k}-\bfX^*\|}
\]
for a given norm $\|.\|$, where $X^*$ denotes the final bistochastic matrix. 
Since the value of $\xi$ is norm dependent, he called the convergence rate to be linear if for some norm,
$0<\xi<1$, which he proved for Sinkhorn iteration. 
More recently,~\cite{knight} provided a local rate of convergence. 
Due to his work, for classical Sinkhorn iteration the local rate of convergence
of a fully indecomposable matrix, 
is bounded by $\sigma_2^2$ where $\sigma_2$ is the second singular 
value of the bistochastic matrix to which the iteration converges. 
Hence, the following result allows us to estimate the local convergence
rate of Sinkhorn iteration, as $p\to\infty$.

\begin{proposition}
Assume that there is only one optimal permutation. Then, there is
a constant $c>0$ such that
\[ 
1-O(\exp(-cp)) \leq \sigma_2(X(p))\leq 1  \qquad \text{ as} \;p\to \infty
\]
Assume now that the matrix $X(\infty)$ is fully indecomposable (which implies
that there are several optimal permutations). Then, 
\[
\sigma_2(X(p))\to \sigma_2(X(\infty))<1 \qquad\text{ as}\; p\to\infty
\enspace .
\]
\end{proposition}
\begin{proof}
Let $\sigma_1\geq\sigma_2\geq\ldots\geq\sigma_p$ 
and $\tilde{\sigma_1}\geq \tilde{\sigma_2}\geq\ldots\geq\tilde{\sigma_p}$
denote the singular values of two $n\times n$ matrices, $X$ and $\tilde{X}$ respectively.
Define  a diagonal matrix $D$ such that $d_{ii}=\tilde{\sigma_i}-\sigma_i$.
Due to the perturbation theorem of~\cite{Mirsky60} 
for any unitarily invariant norm $\| . \|$ 
we have, $\|D\|\leq\|\tilde{X}-X\|$.
So, for $X(p)$ and $X(\infty)$,
\[
|{\sigma_2}{(X(p))}-{\sigma_2}{(X(\infty))}|\leq \|X(p)-X(\infty)\|_2 \leq  O(\exp(-cp))
\]
for which the constant $c$ depends on the coefficients of the Puiseux series and possibly on the dimension of $X(p)$.
Thus, if the original matrix has only one optimal permutation, ${\sigma_2}{(X(\infty))}=1$ which implies that 
\[
1- O(\exp(-cp))\leq{\sigma_2}{(X(p))}
\] 
Moreover according to the Birkhoff-von Neumann theorem~\cite{birkhoff46}, 
for any norm $\|. \|$ on $\R^n$ which is invariant under permutation 
of the coordinates and for any bistochastic matrix $X$, $\|X\|=1$ and subsequently
\[
1-O(\exp(-cp)) \leq \sigma_2(X(p))\leq 1 
\] 

When $X(\infty)$ is fully indecomposable, 
since the multiplication of two fully indecomposable matrices is also fully 
indecomposable, $M=X(\infty)X^T(\infty)$ is fully indecomposable. 
Note also that for all $1\leq i\leq n$,
$m_{ii}=\sum_{j=1}^n x_{ij}^2>0$, which implies that $M$ is primitive.
Then, according to the Perron-Frobenius theorem, all the eigenvalues of $M$ 
distinct from $\rho(M)$ have a modulus
strictly smaller than $\rho(M)=1$ which yields $\sigma_2(X(\infty))<1$.  
\end{proof}

\subsection{Definition of deformed Sinkhorn iteration}

Let $A\in\R^{n\times n}$ be a real non-negative matrix.
The standard Sinkhorn iteration is defined as follows
\begin{eqnarray*}
&&Z_0=\ROWS (A)\enspace;\\
&&W_{m}=\COLS(Z_{m-1})\enspace;\\
&&Z_{m}=\ROWS(W_{m})\enspace;\\
\end{eqnarray*}
where $W_{m}$ and $Z_{m}$ respectively, are column scaled and row scaled matrices 
and $\COLS$ denote the column scaling operator in which all the columns 
of a matrix are divided by their sums and $\ROWS$ be the similar operator for rows. 
It is easy to verify that, $\ROWS(DB)=\ROWS(B)$ and $\COLS(BD)=\COLS(B)$ for any diagonal matrix $D$. 
Now consider the following iteration for a sequence of vectors $u_m,v_m\in \R^n$ 
\begin{eqnarray}
\label{eq:dad1}
&&v_0=\unit\\
\label{eq:dad2}
&&u_{m+1}=\I(Av_m)\\
\label{eq:dad3}
&&v_{m+1}=\I(A^Tu_{m+1})
\end{eqnarray}
where $\unit$ denotes the vector $(1,1,\ldots,1)^T$ of dimension $n$ and $\I$ 
denotes the operator which inverts every entry of a vector. 
In the sequel, for all vectors $v=(v_1,\ldots,v_n)^T\in\R^n$,
$\diag(v)$ denotes the diagonal matrix with diagonal entries
$v_1,\dots,v_n$.

\begin{proposition}
For a non-negative matrix $A$, which has total support, the iteration 
defined by Equations~\ref{eq:dad1},~\ref{eq:dad2} and~\ref{eq:dad3} coincides with Sinkhorn iteration
such that 
\begin{eqnarray*}
&&W_{m}=\diag(u_{m})A\diag(v_{m})\\
&&Z_{m}=\diag(u_{m+1})A\diag(v_m)\\
\end{eqnarray*}
\end{proposition}

\begin{proof}
Note that $Z_0=\ROWS(A)$ and that $Z_{m}=\ROWS(A\diag(v_{m}))$ so 
\begin{equation*}
Z_{m}=\ROWS(\diag(u_{m})A\diag(v_{m}))=\ROWS(W_m)\\
\end{equation*}
a similar statement can be proved for $W_m$ to show that $W_m=\COLS(Z_{m-1})$.
\end{proof}

Now consider the following iteration which is a standard Sinkhorn iteration 
with a deformation of using an increasing sequence $p_m$ which goes to infinity. 
\begin{eqnarray*}
&&v_0=\unit\\
&&u_{m+1}=\I(A^{(p_{m+1})}v_m)\enspace;\\
&&v_{m+1}=\I(A^{(p_{m+1})T}u_{m+1})\enspace.
\end{eqnarray*}
Analogous to the standard Sinkhorn iteration, let 
$W_{m+1}$ and $Z_{m}$ respectively, be column scaled and row scaled matrices defined as the following:
\begin{eqnarray}
&&W_{m}=\diag(u_{m})A^{(p_{m})}\diag(v_{m})\nonumber \\
\label{eq:zdefinition}
&&Z_{m}=\diag(u_{m+1})A^{(p_{m+1})}\diag(v_m)
\end{eqnarray}
\begin{proposition}
\label{rowcol_scal_prp}
For a diagonal matrix $D$, real matrices $B,C$ and the matrices $W_{m},Z_{m}$ in the iteration, the following properties hold. 
\begin{enumerate}
\item $\ROWS(C \circ (DB))=\ROWS(C \circ B)$ where $ \circ $ indicates the Hadamard product
\item $W_{m}=\COLS(Z_{m-1})$
\item $Z_{m}=\ROWS(W_{m} \circ A^{(p_{m+1}-p_m)})$
\end{enumerate}
\end{proposition}
\begin{proof}
We only prove the last one since others are straightforward. 
\begin{eqnarray*}
Z_{m}
&&=\ROWS(A^{(p_{m+1})}\diag(v_{m}))\\
&&=\ROWS(A^{(p_m)}\diag(v_{m}) \circ A^{(p_{m+1}-p_m)} )\\
&&=\ROWS(( \diag(u_{m})A^{(p_m)}\diag(v_{m})) \circ A^{(p_{m+1}-p_m)} )\\
&&=\ROWS(W_{m} \circ A^{(p_{m+1}-p_m)} )
\end{eqnarray*}
\end{proof}

So we define \textit{deformed Sinkhorn iteration} as the following
\begin{eqnarray}
&&Z_0=\ROWS(A^{(p_1)})\enspace; \nonumber\\ 
\label{colscale}
&&W_m=\COLS(Z_{m-1}), \quad c_m=({Z_{m-1}}^T)\unit \enspace;\\
\label{rowscale}
&&Z_m=\ROWS(W_m  \circ  A^{(p_{m+1}-p_m)}),\quad r_m=(W_m  \circ  A^{(p_{m+1}-p_m)})\unit \enspace.
\end{eqnarray}
Here, $r_m,c_m$ respectively denote the vectors of row sums and column sums. 

In the following two sections, we will prove that 
the deformed Sinkhorn iteration will converge to a bistochastic matrix 
where all the nonzero entries belong to an optimal permutation of the original matrix.

\subsection{Convergence to optimal assignment} 
\label{sec:defopt}
For an input matrix, $A=(a_{ij})$, assume that the deformed Sinkhorn iteration 
converges to a bistochastic matrix.
 Define the weight of a permutation, $\sigma$, with respect to $A$, to be $\omega_\sigma(A)=\prod_ia_{i\sigma(i)}$.
If $A$ has a support, it should have at least
one optimal permutation as $\sigma_{opt}$ with nonzero weight.
It is evident that $\sigma_{opt}$ is the optimal permutation for all 
the matrices $W_m$ and $Z_m$ produced by each deformed Sinkhorn iteration.
Observe that for all permutations $\sigma$ and $\pi$, 
the ratio $\frac{\omega_\sigma(A)}{\omega_\pi(A)}$ is invariant if we multiply the matrix $A$ by diagonal matrices. So it follows from the Equation~\ref{eq:zdefinition} that
\[
\gamma_m=\frac{\omega_\sigma(Z_m)}{\omega_\pi(Z_m)}=\gamma_{m-1}(\frac{\omega_\sigma(A)}{\omega_\pi(A)})^{p_{m+1}-p_m}=(\frac{\omega_\sigma(A)}{\omega_\pi(A)})^{p_{m+1}}
\]
Thus, for all non optimal permutations such as $\sigma$,
$\frac{\omega_\sigma(Z_m)}{\omega_{\sigma_{opt}}(Z_m)}$ 
will converge to zero when $p_m\rightarrow\infty$. 
Since in each iteration 
the weight of optimal permutation, $\omega_{\sigma_{opt}}(Z_m)$,
is bounded above by $1$,
the weight of all non optimal permutations will converge to zero which yields the following lemma.
\begin{lemma}
Assume that the deformed Sinkhorn iteration converges to a matrix, $Z$, 
produced by the deformed Sinkhorn iteration when $p_m\rightarrow\infty$. If the original 
matrix $A$ has a support, then 
all the permutations of $Z$ have zero weight, except the optimal 
permutations of the original matrix $A$.
\end{lemma}

Due to the theorem of Birkhoff-von Neumann, a square bistochastic matrix in $\R$ 
is a convex combination of permutation matrices. Hence, all the nonzero 
entries of a bistochastic matrix belong to a permutation with nonzero weight. This statement 
together with the previous lemma yield the following theorem.
\begin{theorem}
\label{optimalassign}
For a non-negative matrix $A$ which has a support, as $p_m\rightarrow\infty$, 
if the \textit{deformed Sinkhorn iteration} converges to a matrix $X$, 
then all the nonzero entries of $X$ belong to an optimal permutation of the original matrix.
\end{theorem}

\subsection{Convergence to bistochastic matrix for positive matrices}
\label{sec:deformed_conv}
Recall that the rate of convergence of the classical Sinkhorn iteration is bounded 
above by $\kappa(A)^2$ where
$\kappa(A)=\frac{\theta(A)^{1/2}-1}{\theta(A)^{1/2}+1}$.
The following theorem presents the main result of this section:
\begin{theorem}
\label{logconvergencepositive}
Let $A$ be a positive matrix. If $p_m=a\log (m+1)$ where $0<a\log\theta<2$, then the 
deformed Sinkhorn iteration will converge to a bistochastic matrix 
and subsequently to a solution of optimal assignment of the original matrix $A$.
\end{theorem}

The proof relies on the next lemmas.
For a matrix $A$, $\theta(A)=\theta(A^T)$,
and for two diagonally equivalent matrices such as $A$ and $B$, $\theta(A)=\theta(B)$. 
\begin{lemma}
For positive matrices $A$ and $B$, diagonal matrix $D$ and the Hilbert projective metric
$d(x,x')$, the following properties hold.
\begin{enumerate}
\item $d(Ax,Ax')\leq \kappa(A)d(x,x')$
\item $d((A\circ B)x,x')\leq \log{\frac{\max(B)}{\min(B)}}+ d(Ax,x')$
\item $\kappa(AD \circ B)=\kappa(A \circ BD)=\kappa((A \circ B)D)=\kappa(D(A \circ B))=\kappa(A \circ B)$
\end{enumerate}
\end{lemma}

\begin{proof}
The proof is straightforward.
\end{proof}
\begin{corollary}
$\kappa(A)$ is invariant under $\ROWS$ or $\COLS$ operators.
\end{corollary}

\begin{lemma}
Let $W_m$ and $Z_m$ be the matrices in Equations~(\ref{colscale},\ref{rowscale}) at iteration $m$.
The following properties hold.
\begin{enumerate}
\item $\kappa(Z_m)=\kappa(A^{(p_{m+1})})$
\item $\kappa(W_m)=\kappa(A^{(p_m)})$
\end{enumerate}
\end{lemma}

\begin{proof}
The proof is straightforward.%
\end{proof}

The next lemma is analogous to Lemma~2 in~\cite{Frank89},
where the classical Sinkhorn iteration is considered.
\begin{lemma}
\label{convergencespeed}
Let $r_m,c_m$ be the vectors defined in Equation~(\ref{colscale},\ref{rowscale}) at iteration $m$ and $M=\frac{\max(A)}{\min(A)}$ 
then,

\begin{eqnarray*}
\label{mainineq1}
d(r_m,\unit)
&\leq & 
(p_{m+1}-p_{m})\log{M}+(p_{m}-p_{m-1})\kappa(A^{(p_m)})\log{M} \nonumber \\
&& +\kappa(A^{(p_m)})\kappa(A^{(p_{m-1})})d(r_{m-1},\unit)\\ 
\label{mainineq2}
d(c_m,\unit)
&\leq&
(p_{m}-p_{m-1})\log M+ (p_{m}-p_{m-1})\kappa(A^{(p_{m-1})})\log M \nonumber \\
&& + \kappa^2(A^{(p_{m-1})})d(c_{m-1},\unit)
\end{eqnarray*}
\end{lemma}

\begin{proof}
We have,
\begin{eqnarray*}
r_m
&&=(W_m \circ A^{(p_{m+1}-p_m)})\unit=(Z_{m-1} \diag(\I(c_m)) \circ A^{(p_{m+1}-p_m)})\unit
\\&&=(Z_{m-1} \circ A^{(p_{m+1}-p_m)})\diag(\I(c_m))\unit=(Z_{m-1} \circ A^{(p_{m+1}-p_m)})(\I(c_m)),
\end{eqnarray*}
so 
\begin{eqnarray*}
d(r_m,\unit)
&&
=d((Z_{m-1} \circ A^{(p_{m+1}-p_m)})(\I(c_m)),Z_{m-1}\unit)\\
&&\leq (p_{m+1}-p_m)\log M+\kappa(Z_{m-1})d(c_m,\unit)\\
&&= (p_{m+1}-p_m)\log M+\kappa(A^{(p_m)})d(c_m,\unit).
\end{eqnarray*}
Also 
\begin{eqnarray*}
d(c_m,\unit)
&&=d(({W_{m-1}}^T \circ A^{(p_m-p_{m-1})T})(\I(r_{m-1})),{W_{m-1}}^T\unit)\\
&&\leq (p_m-p_{m-1})\log M+ \kappa({W_{m-1}}^T)d(\I(r_{m-1}),\unit)\\
&&= (p_m-p_{m-1})\log M+\kappa({W_{m-1}})d(r_{m-1},\unit)\\
&&= (p_m-p_{m-1})\log M+\kappa({A^{(p_{m-1})}})d(r_{m-1},\unit),
\end{eqnarray*}
then
\begin{eqnarray*}
d(r_m,\unit)
&\leq & 
(p_{m+1}-p_{m})\log{M}+(p_{m}-p_{m-1})\kappa(A^{(p_m)})\log{M} \nonumber \\
&& +\kappa(A^{(p_m)})\kappa(A^{(p_{m-1})})d(r_{m-1},\unit)
\end{eqnarray*}
The second statement is established in a similar way.
\end{proof}

\begin{lemma}
Assume that $p_m=a\log(m+1)$, where $0<a\log{\theta(A)}<2$. Then we have 
\[
\lim_{m\rightarrow\infty}{d(c_m,\unit)}=0\enspace.
\]
\end{lemma}
\begin{proof}
Since 
\begin{eqnarray*}
d(c_m,\unit)
&=& a\log\frac{m+1}{m}\log M +a\log\frac{m+1}{m}\kappa(A^{(p_{m-1})})\log M\\
&&+\kappa^2(A^{(p_{m-1})})d(c_{m-1},\unit)\\
&<& \frac{2a\log M}m + \kappa^2(A^{(p_{m-1})})d(c_{m-1},\unit) \enspace.
\end{eqnarray*} 
Let $\beta_1:=d(c_1,\unit)$, and define the sequence $\beta_m$
by $\beta_m:=f_{m-1}(\beta_{m-1})$, 
where
\[
f_{m-1}(x)=\frac{2a\log M}{m}+\kappa^2(A^{(p_{m-1})})x \enspace .
\]
Since every function $f_m$ is nondecreasing, an induction
shows that $d(c_m,\unit)\leq \beta_m$, for all $m \geq 1$, and so,
it suffices to show that $\lim_m \beta_m =0$.

Let $l_m$ be the fixed point of $f_{m-1}$.
Setting
$
\alpha:=
{\frac{a\log\theta(A)}{2}}
$
and observing that
\[
1-\kappa^2(A^{(p_{m-1})})= \frac{4m^{-\alpha}}{(1+m^{-\alpha})^2}\enspace,
\]
we get
\begin{equation*}
l_m
=\frac{2a\log M}{m(1-\kappa^2(A^{(p_{m-1})}))}\\ 
= \frac{a\log M}{2} \frac{\big(1+m ^{-\alpha}\big)^2}{m^{1-\alpha} }
\enspace .
\end{equation*}
Since $0<\alpha<1$, one readily checks that the sequence $l_m$ decreases
with $m$ and converges to zero.
If $\beta_{m+1}\leq l_{m}$ for every $m$, then $\lim_{m\rightarrow\infty}\beta_m
\leq \lim_{m\rightarrow \infty}l_m=0$, and the result is established.
Assume now that $\beta_{m+1}>l_m$ for some $m$. Define
$\delta_{k}:=\beta_{k+1}-l_k$ for all $k\geq m$. Observe
that
\begin{align*}
\delta_{k+1}
&=f_{k}(\beta_k)-f_k(l_k)= \kappa^2(A^{(p_k)})(\beta_k-l_k)\\
&= \kappa^2(A^{(p_k)})\delta_k + \kappa^2(A^{(p_k)})(l_{k-1}-l_k)
\enspace .
\end{align*}
Using the fact that $\kappa^2(A^{(p_r)})\leq 1$ holds for all $r$,
an immediate induction yields
\begin{align}
\delta_k \leq\big(\prod_{r=m}^{k-1}\kappa^2(A^{(p_r)})\big)\delta_m+l_{m}-l_k ,
\qquad \forall k\geq m+1 \enspace .
\label{e-deltak}
\end{align}
Since $1-\kappa^2(A^{(p_r)})\sim 4 r^{-\alpha}$, we have
\[
\prod_{r=m}^\infty \kappa(A^{(p_r)}) =0 
\]
Letting $k\to \infty$ in~\eqref{e-deltak}, we get
$\limsup_{k\to\infty}\delta_k \leq l_m$.
Since this holds for all $m$, it follows that
$\limsup_{k\to\infty}\delta_k \leq 0$, and so,

\[
\limsup_{k\to\infty}\beta_{k+1}=\limsup_{k\to\infty}
\delta_k+l_k  \leq 
\limsup_{k\to\infty}\delta_k
+ \lim_{k\to\infty} l_k =0\enspace.
\]
Hence, $\beta_k$ converges to zero.
\end{proof}

The proof of Theorem~\ref{logconvergencepositive} is achieved since 
$\lim_{m\rightarrow\infty}{d(c_m,\unit)}=0$ implies that $\lim_{m\rightarrow\infty}{d(r_m,\unit)}=0$.

\section{Preprocessing for the optimal assignment problem}
\label{sec:fixp}

In this section we introduce a new algorithm which can be used as
a parallel preprocessing algorithm to solve large dense 
optimal assignment problems, in order
to delete the entries not belonging to optimal assignment.
This approach is based on computing $X(p)$, defined in Theorem~\ref{thm:prop-converge1},
as an approximation to $X(\infty)$ for relatively large values of $p$.
Before going further we need to address the numerical instability of 
computing $X(p)$.
 
\subsection{Avoiding numerical instability }
\label{sec:numins}
To compute $X(p)$, one has to start from the input matrix $A^{(p)}$ and compute the
bistochastic matrix by 
applying a numerical algorithm such as Sinkhorn iteration or Newton method.
But the naive computation of $A^{(p)}$ is numerically unstable for 
large values of $p$. In the following we provide two approaches 
to avoid this numerical instability.
The first method is 
a prescaling step which can be followed by any scaling algorithm such as the
Sinkhorn iteration or Newton method.
A limitation of this method  is that the value of $p$ cannot
exceed $\ln l$ where $l$ is the largest number, in the numerical range (for example $700$ in
double-precision floating-point arithmetic).
This is overcome by the second approach, in which
Sinkhorn iteration is implemented with ``log-glasses''
(along the lines of tropical geometry). This allows one to 
arbitrarily increase the value of $p$. However, 
this approach does not naturally carry over
to other (non-Sinkhorn) scaling algorithms.

\subsubsection{Prescaling step for any scaling algorithm}
\label{sec:prescal}
To avoid the numerical instability one can use the prescaling step presented below.
We set $\max(A)=\max_{ij}a_{ij},\min(A)=\min_{a_{ij}>0}a_{ij}$. 
By applying this prescaling, all the nonzero scaled entries will be placed in $[1,e]$ interval
where $e$ is Napier's constant.
In the case when $\max(A)/\min(A)>e$, the prescaling has another interesting property, 
that is, the scaled matrix is invariant by any entrywise power of the input matrix. 
In other words, 
if we apply the prescaling to the matrix $A^{(q)}$, for all $q\geq 1$, the matrix
obtained after the prescaling step turns out to be independent of the choice of $q$.
When $\frac{\max(A)}{\min(A)}<e$ the entries of $A$ have already been located in
the interval $\min(A)[1,e]$, then we do not need to perform the previous prescaling since
the denominator in the formula defining $m$ will be small if $\max(A)$ is close to $\min(A)$.

\begin{algorithmic}[1]
\newcommand{\algorithmicfunction}{\textbf{function}\ }
\newcommand{\algorithmicfunctionend}{\textbf{end function}\ }
\newcommand{\algorithmicassume}{\textbf{assume: }}
\newcommand{\algorithmicoutput}{\textbf{output: }}
\newcommand{\algorithmiccall}{\textbf{call }}
\IF {$\frac{\max(A)}{\min(A)}>e$}  
	\STATE $m\gets \frac{1}{\log(\max(A)/\min(A))}$ 
	\STATE $c\gets e^{\frac{\log(\min(A))}{\log(\max(A)/\min(A))}}$
	\STATE $A\gets \frac{1}{c}A^{(m)}$	
\ELSE	
	\STATE $A\gets \frac{1}{\min(A)}A$
\ENDIF
\end{algorithmic}

\subsubsection{Logarithmic $p$-Sinkhorn iteration}
\label{subsec:logpsinkh}
The prescaling which has been proposed in the previous section 
has a theoretical disadvantage i.e. the increase of $p$ is 
limited to $\ln l$ where $l$ is the largest number in the numerical range.
We next give a log-coordinate implementation of Sinkhorn iteration
which avoids this limitation. 

Consider the Sinkhorn iteration which was 
defined by Equations~\ref{eq:dad1},~\ref{eq:dad2} and~\ref{eq:dad3}
and let $\bar{u}_m=p^{-1} \log{u_m}$ and $\bar{v}_m=p^{-1} \log{v_m}$, 
be the logarithmic values of the vectors $u_m,v_m$.
The logarithmic form of this iteration can be written as:
\begin{eqnarray*}
&&{\bar{u}{}_{m+1}}_i=-\frac{1}{p}\log{\sum_{j}{\exp{p(\log a_{ij}+{\bar{v}{}_m}_j)}}}\\
&&{\bar{v}{}_{m+1}}_i=-\frac{1}{p}\log{\sum_{j}{\exp{p(\log a_{ji}+{\bar{u}{}_{m+1}}_j)}}}
\end{eqnarray*}
Let 
\begin{eqnarray*}
&&\hat{x}_{ij}=\log{a_{ij}}+{\bar{v}_m}{}_j-\max_j{(\log a_{ij}+{\bar{v}_m}{}_j)}\\
&&\hat{y}_{ji}=\log{a_{ji}}+{\bar{u}_{m+1}}{}_j-\max_j{(\log a_{ji}+{\bar{u}_{m+1}}{}_j)}
\end{eqnarray*}
for which $\hat{x}_{ij},\hat{y}_{ji}\leq0$. 
The logarithmic iteration can be reformulated by using $\hat{x}_{ij}$ and $\hat{y}_{ji}$
as the following:
\begin{eqnarray*}
&&{\bar{u}{}_{m+1}}_i=-\max_j{(\log a_{ij}+{\bar{v}{}_m}_j)}-\frac{1}{p}\log{\sum_{j}{\exp{p\hat{x}_{ij}}}}\\
&&{\bar{v}{}_{m+1}}_i=-\max_j{(\log a_{ji}+{\bar{u}{}_{m+1}}_j)}-\frac{1}{p}\log{\sum_{j}{\exp{p\hat{y}_{ji}}}}
\end{eqnarray*}
The last iteration can be computed for a sufficiently large $p$, without having numerical difficulties. 
We note that a related trick was used by \cite{MR1865511} in a different context.

\subsection{Preprocessing algorithm}
\label{sec:preprocalg}

We shall use the term $\epsilon-$bistochastic matrix, meaning that some distance between $X$
and a bistochastic matrix is less than $\epsilon$.
We measure this distance, 
for a column (row) stochastic matrix, that is a matrix for which the sum of all columns (rows) are one, 
by $\max_i{|r_i-1|}$ where $r_i$ indicates the $i$th row (column) sum.

For a fixed $p>0$, the solution for the entropy maximization problem displayed in Equation~(\ref{eq:pentropyopt}) can be computed
by any scaling algorithm such as Sinkhorn iteration or Newton method. 
Using Corollary~\ref{exponential_conv}, 
it can be seen that if the original matrix has only one optimal permutation, the order of magnitude of all 
the entries which belong to the optimal permutation will be $1\pm O(\exp(-cp))$ 
while the order of magnitude of all other entries will be $O(\exp(-cp))$. 
As an example, consider the following $5$ by $5$ random matrix with the bold entries belonging to optimal permutation.
\[A=
\begin{pmatrix}
0.292 & 0.502 & \textbf{0.918} & 0.281 & 0.686  \\ 
0.566 & \textbf{0.437} & 0.044 & 0.128 & 0.153  \\ 
0.483 & 0.269 & 0.482 & \textbf{0.778} & 0.697  \\ 
0.332 & 0.633 & 0.264 & 0.212 & \textbf{0.842}  \\ 
\textbf{0.594} & 0.405 & 0.415 & 0.112 & 0.406  
\end{pmatrix}\]
By applying Sinkhorn iteration on $A^{(50)}$ the following matrix can be computed. 
\[\bfX(50)=\begin{pmatrix}
3.4E-27 				& 1.5E-08 						&\textbf{1.0E+00} & 7.4E-26 			& 4.7E-06  \\ 
4.8E-02 					& \textbf{9.4E-01} & 4.6E-56 					& 4.0E-32 			& 7.9E-28  \\ 
2.5E-13 					& 4.6E-19 				& 9.3E-12 					& \textbf{1.0E+00} & 1.0E-02  \\ 
1.5E-23 					& 1.2E-02				 & 6.2E-27 						& 4.3E-31 				& \textbf{9.8E-01}  \\ 
\textbf{9.5E-01} & 4.1E-02 					& 6.2E-07 					& 1.0E-34 				& 2.3E-06   
\end{pmatrix}\]
Thus, for sufficiently large values of $p$, when $\bfX(p)$ is an
$\epsilon-$bistochastic matrix,
one may delete all the small entries which are less than a threshold $t$, 
chosen consistent with $\epsilon$, while keeping all others.
In this way the size of the original problem in terms of memory requirements will be reduced to
a much smaller one.

Determining a priori the coarsest accuracy $\epsilon$ and the maximal
threshold $t$, which are required to find an optimal permutation
would need to determine the maximal entropy solution $X(\infty)$ characterized
in Theorem~\ref{thm:prop-converge1}. This appears
to be in general a difficult problem.  We choose a different
route, which is to choose a priori $\epsilon$ and $t$ by
a simple heuristic rule,
and then to verify a posteriori that the deletions of small
entries did not alter the value of the optimal assignment
up to a required precision,
thanks to the approximate optimality certificate
described in Proposition~\ref{prop:aoc}. If this is not
the case, then, $\epsilon$ or $t$ must be decreased.
We fix the initial accuracy and threshold by considering the ``worst'' case in which the matrix $X(\infty)$ is uniform,
with all entries equal to $1/n$ (and $n!$ optimal permutations),
leading 
to the conservative choice $\epsilon=t=1/n$.

\begin{proposition}[Approximate optimality certificate]
\label{prop:aoc}
For an input matrix $A$, and a scalar $p$, let $D_r,D_c$ be the diagonal matrices such that 
$X(p)=D_rAD_c$.
Also, let $\sigma_{opt}$ denote an optimal permutation and ${d_r}_i,{d_c}_j$ denote respectively
the $i$th and the $j$th diagonal elements of $D_r,D_c$.
Then,
\begin{equation}
\log(\omega_{\sigma_{opt}}(A))
\leq 
\frac{1}{p}\left(\sum_{i=1}^n \max_{j}\log {x}_{ij}(p)
-\sum_{i=1}^n \log{d_r}_i-\sum_{j=1}^n \log{d_c}_j\right) \enspace . 
\label{e-certificate}
\end{equation}

\end{proposition}
\begin{proof}
Note that for any permutation $\sigma$ and $1\leq i\leq n$ we have 
$p\log a_{i\sigma(i)}=\log {x}_{i\sigma(i)}(p)-\log{d_r}_i-\log{d_c}_{\sigma(i)}$
which yields,
\[
\sum_{i=1}^n p\log a_{i\sigma(i)}=\sum_{i=1}^n\log {x}_{i\sigma(i)}(p)-\sum_{i=1}^n\log{d_r}_i-\sum_{i=1}^n\log{d_c}_{\sigma(i)}\enspace.
\]
Observe that $\sum_{i=1}^n\log{d_c}_{\sigma(i)}=\sum_{j=1}^n\log{d_c}_{j}$
and 
\[
\sum_{i=1}^n\log {x}_{i\sigma(i)}(p)\leq \sum_{i=1}^n\max_j\log {x}_{ij}(p)\enspace,
\]
so we get
\[
\log(\omega_{\sigma}(A))=\sum_{i=1}^n \log a_{i\sigma(i)}\leq
\frac{1}{p}\left(\sum_{i=1}^n\max_j\log {x}_{ij}(p)-\sum_{i=1}^n\log{d_r}_i-\sum_{j=1}^n\log{d_c}_{j}\right)
\]
The latter inequality holds for any permutation which completes the proof.
\end{proof}

For a matrix $A$ and any choice of $p$, and for a threshold $t$,
let us define the matrix $B$ as follows
\begin{equation}
{B}_{ij}=
\begin{cases}
A_{ij} & \mathrm{if} \quad  x(p)_{ij}\geq t,\\
 0 & \mathrm{otherwise};
\end{cases}
\label{eq:matrixred}
\end{equation}
which denotes the reduced matrix after deleting 
some entries of $A$. 
Let $\omega_{\opt}(B)$ denote the value of the optimal assignment of $B$.
Define the ratio $\gamma$ as follows,
\begin{equation}
 \label{eq:optratio}
\gamma=\frac{\exp \frac{1}{p}\left(\sum_{i=1}^n \max_{j}\log {x}_{ij}(p)
-\sum_{i=1}^n \log{d_r}_i-\sum_{j=1}^n \log{d_c}_j\right)}{\omega_{\sigma_{opt}}(B)}\enspace. 
\end{equation}
Note that when the matrix $A$ has only one optimal permutation and when $p$ tends to infinity, 
by Theorem~\ref{thm:prop-converge1}, 
the Inequality~\ref{e-certificate}
will become an equality and 
\[\sigma_{opt}(B)=\sigma_{opt}(A)=
\lim_{p\rightarrow\infty}\frac{1}{p}\left(\sum_{i=1}^n \max_{j}\log {x}_{ij}(p)
-\sum_{i=1}^n \log{d_r}_i-\sum_{j=1}^n \log{d_c}_j\right)\enspace,\]
which yields that $\gamma=1$. 
By comparison with $A$, the matrix $B$ has been sparsified, so that 
it may fit in the memory of a sequential machine in situations in which
$A$ does not. 
In particular, $\sigma_{opt}(B)$ can be computed by a sequential
algorithm. Also, the numerator in the expression of $\gamma$ can be 
readily evaluated. Hence, the approximate optimality certificate
can be used a posteriori to check that the value of an optimal permutation
of $B$ which has been found is close to the value of the optimal
assignment problem for $A$.

\begin{algorithm}                      
\caption{Preprocessing for optimal assignment problem}
\label{alg:alg1}
\begin{algorithmic}
\STATE input: $A,opt\_ratio$	\textbf{comment: }opt\_ratio: is the required optimal ratio.
\STATE Default: $\epsilon,t\gets 1/n,p_0\gets 100$  
\STATE $p\gets p_0$ 
\STATE $\gamma\gets opt\_ratio$ 
\STATE $n\gets size(A,1)$
\WHILE{$\gamma<opt\_ratio$} 
  \STATE \textbf{comment: }Prescaling
  \IF {$\frac{\max(A)}{\min(A)}>e$}  
	  \STATE $m\gets \frac{1}{\log(\max(A)/\min(A))}$, $c\gets e^{\frac{\log(\min(A))}{\log(\max(A)/\min(A))}}$
	  \STATE $A\gets \frac{1}{c}A^{(m)}$	
  \ELSE	
	  \STATE $A\gets \frac{1}{\min(A)}A$
  \ENDIF 
  \STATE \textbf{comment: }Main section
  \STATE apply any diagonal scaling algorithm to $A^{(p)}$ and compute the diagonal matrices
$D_r,D_c$ and $\epsilon-$bistochastic matrix $X$.  
  \STATE Compute the matrix $B$ defined in~\eqref{eq:matrixred}.
  \STATE Compute the value of $\gamma$ defined in~\eqref{eq:optratio}.
  \STATE increase $p$; \textbf{comment: } increase $p$ and 
repeat until the required optimality ratio is achieved.
\ENDWHILE
\STATE return any optimal permutation of the matrix $B$.
\end{algorithmic}
\end{algorithm}

The above arguments lead to Algorithm~\ref{alg:alg1}. 
The inputs of the algorithm are a square matrix and the required ratio of optimality
which we denote by $opt\_ratio$ and which should be greater than $1$. 
We incorporate the prescaling step proposed in Section~\ref{sec:prescal} 
in order to be able to use any scaling algorithm. This algorithm is
justified by Theorem~\ref{thm:prop-converge1}, which implies that
if $p$ is large enough, and if $\epsilon$ and $t$ are sufficiently small,
deleting the small entries will not affect the entries belonging to optimal
permutations.

This is illustrated in Figure~\ref{fig:poptf} 
for a random matrix and 
``lotkin'' matrix from the gallery of Matlab of size $100$.
\begin{figure}[t!]
\begin{minipage}{0.45\linewidth}
\centering
\includegraphics[scale=0.5]{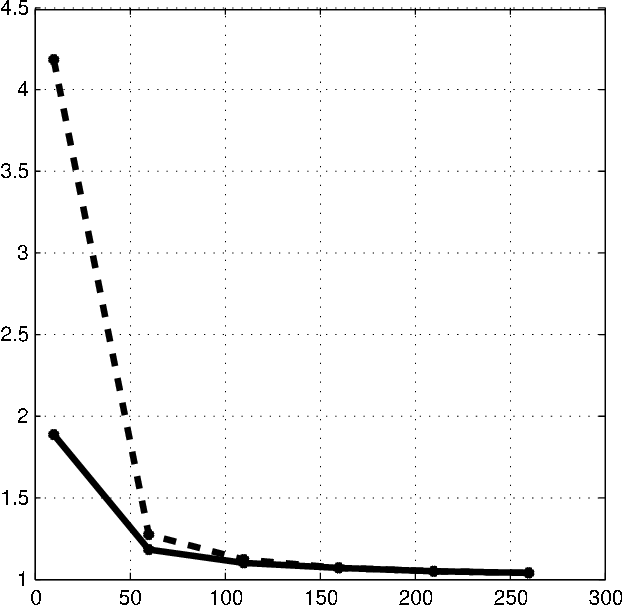} 
\caption{The horizontal axis shows the ratio $\gamma$. The vertical axis shows the values of $p$. 
The solid line gives $\gamma$ as a function of $p$ for a random matrix of size $100$. The dashed line 
gives the same function for the ``lotkin'' matrix from the 
gallery of Matlab of size $100$. We used Sinkhorn iteration as the scaling algorithm. 
} 
\label{fig:poptf}
\end{minipage} 
 \hfill
\begin{minipage}{0.49\linewidth}
\centering
\includegraphics[scale=0.5]{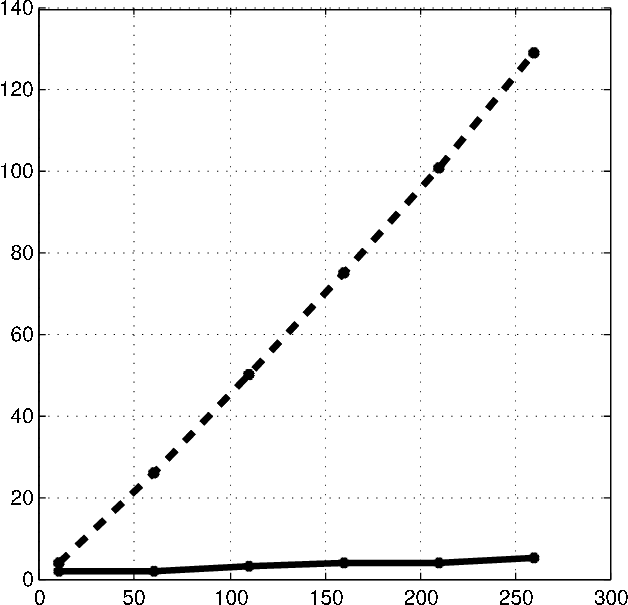} 
\caption{
The horizontal axis shows the number of iterations 
to a $1/100$-bistochastic matrix by using Sinkhorn iteration.
The vertical axis shows the values of $p$. 
The solid line gives the 
number of iterations for a random matrix of size $100$ and
as a function of $p$. The dashed line 
shows the same function for the ``lotkin'' matrix from the 
gallery of Matlab of size $100$. 
} 
\label{fig:iteration}
\end{minipage} 
\end{figure}
Increasing $p$ leads to the true value of the assignment problem
($\gamma$ tends to one), but the convergence of the Sinkhorn algorithm
becomes slower as $p$ increases.
This fact is illustrated in Figure~\ref{fig:iteration}.
Here the required number of Sinkhorn iterations is demonstrated 
for several values of $p$ for a random matrix and ``lotkin'' matrix, both of 
size $100$. 
We set $p_0=100$ by default,
together with
$\epsilon=t=1/n$ as pointed
out above. We increased $p$ by $50$ at each step when $\gamma$ is larger than a given optimal ratio.

\section{Experimental results}
\label{sec:experminets}
In this section we provide experiments that show the efficiency of Algorithm~\ref{alg:alg1}.
We use Sinkhorn iteration and Newton iteration 
as a scaling algorithm in the course of Algorithm~\ref{alg:alg1}. 
In our experiments we set the required optimal ratio ($opt\_ratio$ in Algorithm~\ref{alg:alg1})
around $2$. We present the exact value of $\gamma$ and 
the value of $p$ corresponding to each input matrix in our experimental results.
Note that to compute $\gamma$ we need to compute the optimal assignment after deleting the small entries. 
We used two Matlab implementations of Munkres assignment algorithm
downloadable from ``mathworks'' 
website\footnote{\url{http://www.mathworks.com/matlabcentral/fileexchange/6543-functions-for-the-rectangular-assignment-problem/content/assignmentoptimal.m}
\\\url{http://www.mathworks.com/matlabcentral/fileexchange/20328-munkres-assignment-algorithm/content/munkres.m}}.

In our experiments we generated several dense matrices from the 
gallery of test matrices of Matlab. Most of these matrices are full. 
For these experiments we set the dimension to $1000$. Since some of these matrices have
negative entries, before running our algorithm we take the absolute values 
of the matrix entries ($A=abs(A)$).
The experiments are also presented for 
a random matrix, referred to as ``rand'' (the random function of Matlab) and 
an Euclidean random matrix referred to as ``Euclidean'' in the tables below. 
The latter example has been considered in the context of statistical physics~\cite{Parisi2002}.
We draw at random $2n$ points $x_1,\ldots,x_n;y_1,\ldots,y_n$ uniformly
in the unit cube of $\R^3$.
Then, we consider the matrix $A$ where $a_{ij}=\exp(-d(x_i,y_j))$ and $d$ is the Euclidean distance.
In this way, the permutation $\sigma$ which maximizes $\prod_{i=1}^n{a_{ij}}$ is the same permutation 
which minimizes the distance between these two sets of points.

The columns of Tables~\ref{tbl:dense},~\ref{tbl:newtondensesymmetric} 
and~\ref{tbl:newtondensenonsymmetric} from left to right are: gallery name, 
the value of $\gamma$, 
the value of $p$, 
number of iterations and the percentage of remaining entries over 
the number of nonzero entries after applying the preprocessing algorithm.
We used Matlab version 7.12.0.

\subsection{Sinkhorn iteration}
\label{subsec:sinkhexperiments}
The experiments presented in Table~\ref{tbl:dense} 
are obtained by using Sinkhorn iteration as a diagonal scaling method in Algorithm~\ref{alg:alg1}.
For most of the cases the value of $\gamma$ is already less than $2$
when $p=100$
which means that, at worst, we loose the optimality by a factor of $2$,
whereas for some matrices, we do get the optimal permutation ($\gamma=1$).
Note that for ``pei'', ``randcorr'' and ``toeppd'' matrices,
the only nonzero entries are those that belong to the optimal permutation.

For more than $50\%$
of the cases the original problem is reduced to a new problem which has less 
than $4\%$ of the original entries and in $81\%$
it is reduced to a new problem with less than $27\%$ of the input entries.
The algorithm generally converges quickly to the solution. 
More precisely in $37\%$ of the cases, the algorithm converges in at most $2$ iterations 
and in $62\%$  of the cases, the algorithm converges in at most $132$ iterations. 
Since, Sinkhorn iteration
can be implemented in parallel, this method can be efficiently applied to large dense 
optimal assignment problems as a parallel preprocessing to reduce the size of the original problem.

\begin{table}[t!]
\tblcaption{Sinkhorn iteration for dense matrices from the 
gallery of test matrices of Matlab and for random and random Euclidean distance matrices}
{%
\begin{tabular}{@{}ccccc@{}}
\tblhead{Gallery &  $\gamma$ & $p$ & No.it. &  Rem.(\%)}
cauchy & $1.490$ & $100$ & $70$ & $46.17$ \\
minij & $1.000$ & $100$ & $568$ & $24.05$ \\ 
moler & $1.028$ & $100$ & $281$ & $26.24$ \\ 
pei & $1.000$ & $100$ & $1$ & $0.10$ \\ 
prolate & $1.000$ & $100$ & $16$ & $1.69$ \\ 
randcorr & $1.000$ & $100$ & $1$ & $0.10$ \\ 
toeppd & $1.000$ & $100$ & $1$ & $0.10$ \\ 
chebvand & $1.745$ & $150$ & $2$ & $31.78$ \\ 
circul & $1.000$ & $100$ & $1$ & $17.20$ \\ 
cycol & $1.900$ & $450$ & $93$ & $1.36$ \\ 
rand & $1.839$ & $100$ & $2$ & $25.87$ \\ 
euclidean & $1.728$ & $200$ & $1416$ & $0.92$ \\ 
chebspec & $1.004$ & $100$ & $343$ & $3.54$ \\ 
lehmer & $1.000$ & $100$ & $858$ & $16.46$ \\ 
gcdmat & $1.000$ & $100$ & $2405$ & $0.20$ \\ 
lotkin & $1.817$ & $200$ & $132$ & $40.72$
\lastline
\end{tabular}
}%
\label{tbl:dense}
\end{table}

\subsection{Newton iteration}
\label{subsec:newtonit}
For the sake of comparison, we implemented the preprocessing algorithm by
calling a Newton algorithm at each step.
Solving the diagonal matrix scaling problem by using Newton iteration has been considered first
in the work of~\cite{kachiyan92} for positive semidefinite symmetric matrices.
They have considered the more general problem of finding a positive zero of the mapping
\[
f(x)=b+Ax-x^{-1}
\]  
where $A$ is a given matrix of dimension $n$ and $b$ is a fixed $n-$dimensional vector. They proposed a path-following 
Newton algorithm of complexity $O(\sqrt{n}L)$ where $L$ is the binary length of the input.

Recently, Knight and Ruiz have considered a Newton algorithm for non-negative matrices~\cite{knight2007}.
For a symmetric matrix $A$, they considered the diagonal matrix scaling problem as finding a vector $x$ such that 
\[
f(x)=D(x)Ax-\unit=0
\]
where $D(x)=\diag(x)$. If $A$ is nonsymmetric, then the following matrix will be considered 
as the input of the algorithm.
\[
S=\begin{pmatrix}
0 & A \\
A^T & 0
\end{pmatrix}
\] 
They showed that Newton iteration can be written as 
\[
A_kx_{k+1}=Ax_k+D(x_k)^{-1}\unit
\]
where $A_k=A+D(x_k)^{-1}D(Ax_k)$. Thus in each iteration a linear system of equations 
should be solved for which they used the Conjugate Gradient method. 
In the nonsymmetric case, the latter linear system is singular, however
it is proved that the system is consistent whenever $A$ has support.
Our experiments which will be presented later show that, 
the method works quickly for dense nonsymmetric matrices. 
More details and the exact
implementation of this method can be found in~\cite{knight2007}.

Here, we used the latter method to find the scaling matrices in Algorithm~\ref{alg:alg1}. 
In Tables~\ref{tbl:newtondensesymmetric} and~\ref{tbl:newtondensenonsymmetric}, 
No.it. denotes the total number of 
operations, each of them takes $O(n^2)$ time to be done. 
This includes all the iterations of Conjugate 
Gradient method for each Newton step.
Tables~\ref{tbl:newtondensesymmetric} and~\ref{tbl:newtondensenonsymmetric} 
show the results for dense symmetric and nonsymmetric matrices.
For both cases the algorithm converges rapidly in small number of iterations.
The percentage of the remaining entries is reasonably less than the original problem. 
For ``pei'', ``randcorr'' and ``toeppd'' matrices,
the only nonzero entries are those that belong to the optimal permutation.
Also, in more than $42\%$ of the cases, the original problem 
is reduced to a much smaller problem which has less than $4\%$ of the original entries
and in $68\%$ of the cases the problem is reduced to a problem with less than $27\%$ of the original entries.

\begin{table}[t!]
\tblcaption{Newton iteration for dense symmetric matrices}
{%
\begin{tabular}{@{}ccccc@{}}
\tblhead{Gallery &  $\gamma$ & $p$ & No.it. &  Rem.(\%)}
cauchy &  $1.4907$ & $100$ & $155$ &  $46.17$\\
gcdmat & $1.000$ & $100$ & $151$ & $0.20$ \\ 
lehmer & $1.000$ & $100$ & $162$ & $16.46$ \\ 
minij & $1.000$ & $100$ & $162$ & $24.05$ \\ 
moler & $1.028$ & $100$ & $161$ & $26.24$ \\ 
orthog & $1.000$ & $100$ & $161$ & $48.01$ \\ 
pei & $1.000$ & $100$ & $151$ & $0.10$ \\ 
prolate & $1.000$ & $100$ & $155$ & $1.69$ \\ 
randcorr & $1.000$ & $100$ & $151$ & $0.10$ \\ 
toeppd & $1.000$ & $100$ & $151$ & $0.10$ \\ 
fiedler & $1.711$ & $100$ &  $170$ & $33.77$
\lastline
\end{tabular}
}
\label{tbl:newtondensesymmetric}
\end{table}

\begin{table}[t!]
\tblcaption{Newton iteration for dense nonsymmetric matrices}
{%
\begin{tabular}{@{}ccccc@{}}
\tblhead{Gallery &  $\gamma$ & $p$ & No.it. &  Rem.(\%)}
chebspec & $1.001$ & $100$ & $214$ & $3.54$ \\ 
chebvand & $1.744$ & $150$ & $233$ & $31.78$ \\ 
circul & $1.000$ & $100$ & $157$ & $17.20$ \\ 
forsythe & $1.000$ & $100$ & $262$ & $50.00$ \\ 
rand & $1.839$ & $100$ & $163$ & $25.87$ \\ 
euclidean & $1.728$ & $200$ & $742$ & $0.92$ \\ 
cycol & $2.615$ & $350$ & $551$ & $ 1.71$ \\ 
lotkin & $1.817$ & $200$ & $495$ & $40.72$\\
\lastline
\end{tabular}
}
\label{tbl:newtondensenonsymmetric}
\end{table}

\section{Conclusion}
We studied the connection between the optimal assignment problem and the entropy maximization problem, by means of a parametric deformation of the latter.
We proved that, as the deformation parameter goes to infinity,
the matrix maximizing the entropy
converges to a matrix whose nonzero entries are those which belong to optimal assignments. 
This allowed us to develop an iterative method that we
refer to as \textit{deformed-Sinkhorn iteration}.
We proved that the iteration converges to the solution of optimal assignment problem, if 
the input matrix is positive and if it has only one optimal permutation. 
For positive matrices with more than one optimal permutation, the iteration converges to a matrix 
for which all the nonzero entries belong to at least one optimal permutation.

We also proposed an algorithm which
can be used as a preprocessing in the solution of large scale dense optimal assignment problems to reduce the
size of the input problem in terms of memory requirements.
Experimental results have been generated for two variants of the algorithm. 

The first variant, which is based on Sinkhorn iteration,
shows a generally reasonable convergence for dense matrices,
with a reduction of up to $99\%$ of the input size.
This variant can be efficiently used as a parallel preprocessing step to reduce the size of the 
input problem in very large dense optimal assignment problems.
Another variant of the algorithm, implemented by using Newton iteration,
shows generally a faster convergence for the tested matrices.

\section*{Acknowledgement}
 The authors thank Jean-Charles Gilbert for his comments on an early version of this manuscript.

\def\cprime{$'$}

\end{document}